\documentclass[reqno]{amsart}
\usepackage{amsthm,amsmath,amssymb}
\usepackage{stmaryrd,mathrsfs}
\usepackage{esint}
\usepackage{color}

\allowdisplaybreaks

\newcommand{\ud}[0]{\,\mathrm{d}}

\newcommand{\abs}[1]{|#1|}
\newcommand{\Babs}[1]{\Big|#1\Big|}

\newcommand{\Norm}[2]{\|#1\|_{#2}}

\newcommand{\BNorm}[2]{\Big\|#1\Big\|_{#2}}
\newcommand{\pair}[2]{\langle #1,#2 \rangle}

\newcommand{\supp}[0]{\operatorname{supp}}

\newcommand{\R}{\mathbb{R}}

\newcommand{\Z}{\mathbb{Z}}

\newcommand{\prob}[0]{\mathbb{P}}

\def\cyr{\fontencoding{OT2}\fontfamily{wncyr}\selectfont}
\DeclareTextFontCommand{\textcyr}{\cyr}

\swapnumbers \numberwithin{equation}{section}

\theoremstyle{plain}
\newtheorem{theorem}[equation]{Theorem}
\newtheorem{proposition}[equation]{Proposition}
\newtheorem{corollary}[equation]{Corollary}
\newtheorem{lemma}[equation]{Lemma}

\theoremstyle{definition}

\theoremstyle{remark}
\newtheorem{remark}[equation]{Remark}

\makeatletter
\@namedef{subjclassname@2010}{%
  \textup{2010} Mathematics Subject Classification}
\makeatother

\begin{document}

\title[Non-probabilistic proof of the $A_2$ theorem]{Non-probabilistic proof of the $A_2$ theorem, and sharp weighted bounds for the $q$-variation of singular integrals}

\author[T.~Hyt\"onen]{Tuomas P.~Hyt\"onen}
\address{Department of Mathematics and Statistics, P.O.B.~68 (Gustaf H\"all\-str\"omin katu~2b), FI-00014 University of Helsinki, Finland}
\email{tuomas.hytonen@helsinki.fi}
\thanks{T.H. is supported by the European Union through the ERC Starting Grant ``Analytic--probabilistic methods for borderline singular integrals'', and by the Academy of Finland, grants 130166 and 133264.}

\author[M.~T.~Lacey]{Michael T. Lacey}
\address{School of Mathematics \\
Georgia Institute of Technology \\
Atlanta GA 30332 }
\email{lacey@math.gatech.edu}
\thanks{M.L. supported in part by the NSF grant 0968499, and a grant from the Simons Foundation (\#229596 to Michael Lacey). }

\author[C.~P\'erez]{Carlos P\'erez}
\address{
Departamento de An\'alisis Matem\'atico, Facultad de Matem\'aticas,
Universidad De Sevilla, 41080 Sevilla, Spain
}
\email{carlosperez@us.es}
\thanks{C. P. was supported by the Spanish Ministry of Science and Innovation grant MTM2009-08934 and by the Junta de Andaluc\'ia, grant FQM-4745.}

\date{\today}




\begin{abstract}
Any Calder\'on--Zygmund operator $T$ is pointwise dominated by a convergent sum of positive dyadic operators. We give an elementary self-contained proof of this fact, which is simpler than the probabilistic arguments used for all previous results in this direction. Our argument also applies to the $q$-variation of certain Calder\'on--Zygmund operators, a stronger nonlinearity than the maximal truncations. As an application, we obtain new sharp weighted inequalities.
\end{abstract}

\maketitle

\section{Introduction}

The following sharp weighted inequality for Calder\'on--Zygmund operators, known as the $A_2$ theorem, was proved by the first author \cite{Hytonen:A2} in full generality after many intermediate results by others:
\begin{equation}\label{eq:A2}
  \Norm{Tf}{L^2(w)}\lesssim [w]_{A_2}\Norm{f}{L^2(w)}.
\end{equation}
After the first proof, several modifications, simplifications and extensions have appeared. However, all of them have been based on the ``dyadic representation theorem'' from \cite{Hytonen:A2}, namely, the probabilistic formula
\begin{equation}\label{eq:dyRepThm}
  \pair{g}{Tf}=\int_\Omega\sum_{m,n=0}^{\infty}2^{-\max\{m,n\}\alpha/2}\pair{g}{S^{\omega}_{m,n}f}\ud\prob(\omega),
\end{equation}
where $T$ is written as an average of the dyadic shifts $S^\omega_{m,n}$ with respect to random dyadic systems $\mathscr{D}^\omega$ parameterized by a probability space $\Omega$. The proof of this representation is an elaboration of the proof of the $T(1)$ theorem of David--Journ\'e, or more precisely, of the non-homogeneous version of Nazarov--Treil--Volberg \cite{NTV:Tb}.

In this paper, we provide a simpler, non-probabilistic dyadic representation. Strictly speaking, the representation with an exact identity as above is now replaced with a domination, which gives an upper bound for $Tf$ in terms of a series of dyadic shifts. We still need more than one dyadic system, but only a finite collection (in fact, $3^d$) of them suffices. And the proof only uses the kernel bounds in a rather standard way, as well as the weak $(1,1)$ bounds for $T$ as a black box, rather than going through the full machinery of the $T(1)$ theorem, which in effect reproves the boundedness of $T$ as a by-product. This has the advantage that the new approach extends more flexibly to other kinds of operators for which unweighted weak $(1,1)$ bounds are known. In particular, for the first time, we extend \eqref{eq:A2} to the $q$-variation operator $V_q^\phi T$, defined by
\begin{equation}\label{eq:defVar}
   V_q^\phi Tf(x):=\sup_{\{\epsilon_i\}_{i\in\Z}}\Big(\sum_{i\in\Z}\abs{T_{\epsilon_i}^\phi f(x)-T_{\epsilon_{i+1}}^\phi f(x)}^q\Big)^{1/q},
\end{equation}
 where $T_\epsilon^\phi$ is a smooth truncation of a (suitable) Calder\'on--Zygmund operator $T$, and the supremum is over all increasing sequences of positive numbers. Unweighted norm inequalities for these operators have been studied by Campbell et al. \cite{CJRW:Hilbert,CJRW:higher}, among others.

For future considerations, another advantage of the new dyadic domination theorem is that it can also be used to derive estimates for $T$ from the estimates for dyadic shifts in \emph{quasi-normed} spaces like $L^{1,\infty}$: We only have a finite summation over $3^d$ different dyadic systems, which is also well-behaved with respect to the quasi-triangle inequality, unlike the integration over a probability space in the dyadic representation theorem \eqref{eq:dyRepThm}.

In the next two sections, we prove the new dyadic domination theorem for maximal truncations of Calder\'on--Zygmund operators and the $q$-variation of a class of Calder\'on--Zygmund operators, respectively. Corollaries on weighted norm inequalities are presented in the final section.

Let us notice that another ``dyadic domination theorem'', closely related to ours, was almost simultaneously brought out by Lerner \cite{Lerner:posDyad}. There is, however, an important difference: Lerner's result is derived as a \emph{corollary} to the dyadic representation theorem \eqref{eq:dyRepThm} and some non-trivial subsequent developments, whereas our proof is essentially self-contained and provides an \emph{alternative} to \eqref{eq:dyRepThm}. This new approach also applies to the stronger variation operators \eqref{eq:defVar}. See, however, Lerner's paper \cite{Lerner:posDyad} for some different applications.

\section{The dyadic domination theorem for maximal truncations}

Let $T$ be an $\omega$-Calder\'on--Zygmund operator with kernel $K$ having the standard size estimate
\begin{equation*}
  \abs{K(x,y)}\lesssim\frac{1}{\abs{x-y}^d}
\end{equation*}
and the modulus of continuity $\omega$:
\begin{equation*}
  \abs{K(x,y)-K(x',y)}+\abs{K(y,x)-K(y,x')}\leq C\omega\Big(\frac{\abs{x-x'}}{\abs{x-y}}\Big)\frac{1}{\abs{x-y}^d}
\end{equation*}
for $\abs{x-y}>2\abs{x-x'}>0$. Let $\phi$ be a smooth cut-off with $1_{B(0,1)^c}\leq\phi\leq 1_{B(0,\tfrac12)^c}$, and $\chi:=1_{B(0,1)^c}$ the sharp cut-off. Then
\begin{equation*}
  K^\phi_{\epsilon}(x,y):=\phi\Big(\frac{x-y}{\epsilon}\Big)K(x,y),\qquad
  K_\epsilon(x,y):=K^\chi_\epsilon(x,y).
\end{equation*}
Note that $K^\phi_\epsilon$ is also an $\omega$-Calder\'on--Zygmund kernel, uniformly in $\epsilon>0$. Finally, we define the truncated operators
\begin{equation*}
   T^\phi_{\epsilon}f(x):=\int K^\phi_{\epsilon}(x,y)f(y)\ud y,\qquad
  T^\phi_*f(x):=\sup_{\epsilon>0}\abs{T^\phi_{\epsilon}f(x)}
\end{equation*}
and $T_\epsilon:=T^\chi_\epsilon$, $T_*:=T^\chi_*$.

\begin{theorem}\label{thm:dyadicDomination}
Let $T$ be an $\omega$-Calder\'on--Zygmund operator, where the modulus of continuity satisfies the Dini condition
\begin{equation}\label{eq:Dini}
  \int_0^1\omega(t)\frac{\ud t}{t}<\infty.
\end{equation}
Let $Q_0$ be a cube which contains the support of $f$.
Then we have the pointwise bound
\begin{equation*}
    1_{Q_0}T_*f
    \lesssim Mf+\sum_{u\in\{0,\tfrac12,\tfrac13\}^d}\sum_{k=0}^{\infty}\omega(2^{-k})S^u_k\abs{f},
\end{equation*}
where $S^u_k$ is a positive dyadic shift of complexity $k$ with respect to the dyadic system
\begin{equation*}
  \mathscr{D}^u:=\{2^{-j}([0,1)^d+m+(-1)^j u):j\in\Z,m\in\Z^d\}.
\end{equation*}
The dyadic shifts depend upon the choice of $Q_0$ and $f$.  
\end{theorem}

The fact that (the maximal truncation of) a Calder\'on--Zygmund operator is dominated by the maximal operators and positive dyadic shifts was already observed by Hyt\"onen--Lacey \cite{HytLac} and elaborated by Lerner \cite{Lerner:posDyad}. However, their arguments used the random dyadic systems, and were derived from the (non-positive) dyadic representation theorem of \cite{Hytonen:A2}. Here we give a more efficient direct argument. The rest of this section is concerned with the proof of Theorem~\ref{thm:dyadicDomination}.

We first recall the well-known fact that
\begin{equation*}
  \abs{T_*f-T_*^\phi f}\lesssim Mf,
\end{equation*}
and hence it suffices to prove the result for $T_*^\phi$ in place of $T_*$.

We use Lerner's formula \cite{Lerner:formula}:
\begin{equation*}
  1_{Q_0}\abs{T_*^\phi f-m_{T_*^\phi f}(Q_0)}\lesssim M_{\lambda,Q_0}^{\#}(T_*^\phi f)+\sum_{Q\in\mathcal{L}}\omega_{\lambda}(T_*^\phi f,Q^{(1)})1_Q,
\end{equation*}
where, denoting by $g^*$ the non-increasing rearrangement of a measurable $g$,
\begin{equation*}
  \omega_\lambda(g,Q):=\inf_c(1_Q(g-c))^*(\lambda\abs{Q}),\qquad
  M_{\lambda,Q_0}^{\#}g:=\sup_{Q\subseteq Q_0} 1_Q\omega_\lambda(g,Q),
\end{equation*}
and $\mathcal{L}$ is a collection of dyadic subcubes of $Q_0$ with the property that 
\begin{equation}\label{eq:EQ}
  \abs{E(Q)}\geq\tfrac12\abs{Q},\qquad
  E(Q):=Q\setminus\bigcup_{\substack{Q'\in\mathcal{L}\\ Q'\subsetneq Q}}Q'.
\end{equation}
for all $Q\in\mathcal{L}$.

\begin{lemma}
\begin{equation*}
  \omega_\lambda(T_*^\phi f,Q)\lesssim\sum_{k=0}^{\infty}\omega(2^{-k})\fint_{2^k Q}\abs{f}.
\end{equation*}
\end{lemma}

\begin{proof}
For $x\in Q$, we have
\begin{equation*}
\begin{split}
  T_\epsilon^\phi f(x) &=T_\epsilon^\phi (1_{2Q}f)(x)+\Big(T_\epsilon^\phi (1_{(2Q)^c}f)(x)-T_\epsilon^\phi (1_{(2Q)^c}f)(c_Q)\Big) \\
    &\qquad+T_\epsilon^\phi (1_{(2Q)^c}f)(c_Q),
\end{split}
\end{equation*}
and
\begin{equation*}
\begin{split}
  &\abs{T_\epsilon^\phi (1_{(2Q)^c}f)(x)-T_\epsilon^\phi (1_{(2Q)^c}f)(c_Q)} \\
   &=\Babs{\sum_{k=1}^{\infty}\int_{2^{k+1}Q\setminus 2^k Q} [K_\epsilon^\phi (x,y)-K_\epsilon^\phi(c_Q,y)]f(y)\ud y} \\
  &\lesssim\sum_{k=1}^{\infty} \int_{2^{k+1}Q\setminus 2^k Q} \omega\Big(\frac{\abs{x-c_Q}}{\abs{x-y}}\Big)  \frac{1}{\abs{x-y}^d}\abs{f(y)}\ud y \\
  &\lesssim\sum_{k=1}^{\infty}  \omega(2^{-k}) \fint_{2^{k+1}Q}\abs{f}.
\end{split}
\end{equation*}
Hence
\begin{equation*}
\begin{split}
  \abs{T_*^\phi f(x)-T_*^\phi(1_{(2Q)^c}f)(c_Q)}
  &\leq\sup_{\epsilon>0}\abs{T_\epsilon^\phi f(x)-T_\epsilon^\phi (1_{(2Q)^c}f)(c_Q)} \\
  &\leq T_*^\phi(1_{2Q}f)(x)+\sum_{k=1}^{\infty}  \omega(2^{-k}) \fint_{2^{k+1}Q}\abs{f}
\end{split}
\end{equation*}
and thus
\begin{equation*}
\begin{split}
  \omega_\lambda(T_*^\phi f;Q)
  &=\inf_c((T_*^\phi f -c)1_Q)^*(\lambda\abs{Q})
  \leq((T_*^\phi f-T_*^\phi(1_{(2Q)^c}f))1_Q)^* (\lambda\abs{Q}) \\
  &\leq (T_*^\phi(1_{2Q}f))^*(\lambda\abs{Q})+\sum_{k=1}^{\infty}  \omega(2^{-k}) \fint_{2^{k+1}Q}\abs{f},
\end{split}
\end{equation*}
where finally
\begin{equation*}
  (T_*^\phi(1_{2Q}f))^*(\lambda\abs{Q})
  \leq\frac{1}{\lambda\abs{Q}}\Norm{T_*^\phi(1_{2Q}f)}{L^{1,\infty}}
  \lesssim\frac{1}{\abs{Q}}\Norm{1_{2Q}f}{L^1}\lesssim\fint_{2Q}\abs{f}
\end{equation*}
by the weak-type $(1,1)$ inequality for $T_*^\phi$.
\end{proof}

In particular
\begin{equation*}
  M_\lambda^{\#}(T_*^\phi f)
  =\sup_Q 1_Q\omega_\lambda(T_*^\phi f;Q)
  \lesssim\sup_Q 1_Q\sum_{k=0}^{\infty}\omega(2^{-k})\fint_{2^k Q}\abs{f}
  \lesssim Mf,
\end{equation*}
since the Dini condition \eqref{eq:Dini} gives
\begin{equation*}
  \sum_{k=0}^{\infty}\omega(2^{-k})\eqsim\int_0^1\omega(t)\frac{\ud t}{t}<\infty.
\end{equation*}

Next, we need to identify each
\begin{equation*}
  \sum_{Q\in\mathcal{L}} 1_Q\fint_{2^k Q}\abs{f}.
\end{equation*}
as a sum of $3^d$ dyadic shifts acting on $\abs{f}$.

The following lemma is well-known for $k=0$. The version below essentially repeats the same argument.

\begin{lemma}
For any cube $Q$, we can find a shifted dyadic cube $R\in\mathscr{D}^u$ for some $u\in\{0,\tfrac13,\tfrac23\}^d$ such that $Q\subset R$, $2^k Q\subset R^{(k)}$, and $3\cdot\ell(Q)<\ell(R)\leq 6\cdot\ell(Q)$.
\end{lemma}

\begin{proof}
Let $Q=I_1\times\cdots\times I_d$ and $2^k Q=J_1\times\cdots\times J_d$. Choose $2^j$ so that $3\ell(Q)<2^j\leq 6\ell(Q)$, and consider the end-points of the dyadic intervals $2^j([0,1)+m_i+(-1)^j u_i)$, where $m_i\in\Z$ and $u_i\in\{0,\tfrac13,\tfrac23\}$. These form an arithmetic sequence with difference $\tfrac13 2^j>\ell(Q)$. Hence $I_i$ contains at most one such end-point, say with $u_i=v_i$. Similarly, $J_i$ contains at most one end-point of the intervals $2^{j+k}([0,1)+m_i+(-1)^j u_i)$, say with $u_i=w_i$. Consider a remaining value $u_i\in\{0,\tfrac13,\tfrac23\}\setminus\{v_i,w_i\}$. It follows that $I_i$ must be contained in some interval $K_i= 2^j([0,1)+m_i+(-1)^j u_i)$, and similarly $J_i$ must be contained in some interval $L_i = 2^{j+k}([0,1)+m_i'+(-1)^j u_i)$. Since $L_i\cap K_i\supset J_i\cap I_i =I_i$, we have $L_i\supset K_i$ and hence $L_i =K_i^{(k)}$ by the properties of dyadic intervals. Hence, the dyadic cube $R:=K_1\times\cdots\times K_d$ satisfies $R\supset Q$ and $R^{(k)}=L_1\times\cdots\times L_d\supset 2^k Q$, as well as $\ell(R)=\abs{K_i}\leq 6\ell(Q)$.
\end{proof}

Thus we can split $\mathcal{L}$ into $3^d$ subcollections $\mathcal{L}_u$, where each $Q\in\mathcal{L}_u$ has an associated $R=\rho_u(Q)\in\mathcal{D}^u$ with $Q\subset R$, $2^k Q\subset R^{(k)}$, and $3\ell(Q)<\ell(R)\leq 6\ell(Q)$. Since both side-lengths $\ell(Q)$ and $\ell(R)$ are powers of $2$, we have in fact $\ell(R)=4\ell(Q)$. Observe that a given $R\in\mathcal{D}^u$ has at most $4^d$ different preimages $Q\in\mathcal{L}$ under $\rho_u$, as their side-length is determined by $R$, they are all contained in $R$, and pairwise disjoint.

Hence, writing $\mathcal{L}_u^*:=\{\rho_u(Q):Q\in\mathcal{L}_u\}$, we have
\begin{equation*}
\begin{split}
  \sum_{Q\in\mathcal{L}}1_Q\fint_{2^k Q}\phi
  &\lesssim\sum_{u\in\{0,\tfrac13,\tfrac23\}^d}\sum_{Q\in\mathcal{L}_u}1_{\rho_u(Q)}\fint_{\rho_u(Q)^{(k)}}\phi \\
  &\lesssim\sum_{u\in\{0,\tfrac13,\tfrac23\}^d}\sum_{R\in\mathcal{L}_u^*}1_{R}\fint_{R^{(k)}}\phi 
  =:\sum_{u\in\{0,\tfrac13,\tfrac23\}^d}S^u_k\phi.
\end{split}
\end{equation*}

\begin{lemma}
For all $p\in(1,\infty)$, the operator $S^u_k$ satisfies the unweighted estimate
\begin{equation*}
  \Norm{S^u_k\phi}{p}\lesssim\Norm{\phi}{p};
\end{equation*}
in particular, it is a bounded positive dyadic shift.
\end{lemma}

\begin{proof}
Let $\phi\in L^p$ and $\psi\in L^{p'}$. Then
\begin{equation*}
\begin{split}
  \pair{\psi}{S^u_k\phi}
  &=\sum_{R\in\mathcal{L}_u^*}\int_R\psi\cdot\fint_{R^{(k)}}\phi
  =\sum_{R\in\mathcal{L}_u^*}\abs{R}\cdot\fint_R\psi\cdot\fint_{R^{(k)}}\phi \\
  &\leq\sum_{R\in\mathcal{L}_u^*}\abs{R}\cdot\inf_R M\psi\cdot\inf_{R}M\phi \\
  &=\sum_{Q\in\mathcal{L}_u}\abs{\psi_u(Q)}\cdot\inf_{ \rho_u(Q)} M\psi\cdot\inf_{ \rho_u(Q)}M\phi   \\
  &\lesssim\sum_{Q\in\mathcal{L}_u}\abs{E(Q)}\cdot\inf_Q M\psi\cdot\inf_Q M\phi
    \leq\sum_{Q\in\mathcal{L}_u}\int_{E(Q)}M\psi\cdot M\phi \\
   &\leq\int M\psi\cdot M\phi\leq\Norm{M\psi}{p'}\Norm{M\phi}{p}
   \lesssim\Norm{\psi}{p'}\Norm{\phi}{p},
\end{split}
\end{equation*}
where we used the pairwise disjointness of the sets $E(Q)$, $Q\in\mathcal{L}$, and the boundedness of the maximal operator.
\end{proof}

Up to this point, we have seen the following:
\begin{equation*}
\begin{split}
  1_{Q_0}\abs{T_*^\phi f-m_{T_*^\phi f}(Q_0)}
  &\lesssim M^{\#}_{\lambda,Q_0}(T_*^\phi f)+\sum_{Q\in\mathcal{L}}\omega_\lambda(T_*^\phi f;Q^{(1)})1_Q \\
  &\lesssim Mf+\sum_{Q\in\mathcal{L}}\sum_{k=0}^{\infty}\omega(2^{-k})\fint_{2^k Q}\abs{f}\cdot 1_Q \\
  &\lesssim Mf+\sum_{u\in\{0,\tfrac13,\tfrac23\}^d}\sum_{k=0}^{\infty}\omega(2^{-k})S^u_k\abs{f}.
\end{split}
\end{equation*}
It only remains to observe that the median $m_{T_*^\phi f}(Q_0)$ satisfies
\begin{equation*}
  \abs{m_{T_*^\phi f}(Q_0)}\lesssim\frac{1}{\abs{Q_0}}\Norm{T_*^\phi f}{L^{1,\infty}}
  \lesssim\frac{1}{\abs{Q_0}}\Norm{f}{L^1}
  =\fint_{Q_0}\abs{f}\leq\inf_{Q_0}Mf
\end{equation*}
provided that $\supp f\subseteq Q_0$. This completes the proof of Theorem~\ref{thm:dyadicDomination}.


\section{The dyadic domination theorem for variation operators}

As before, let $T_\epsilon^\phi $ and $K_\epsilon^\phi $ be the smooth truncations of the $\omega$-Calder\'on--Zygmund operator $T$ and its kernel $K$. We denote by $\vec{T}^\phi:=\{T_\epsilon^\phi \}_{\epsilon>0}$ and $\vec{K}^\phi:=\{K_\epsilon^\phi \}_{\epsilon>0}$ the corresponding vectorial operator and kernel involving all $\epsilon>0$ simultaneously. Then the previously considered maximal truncation is simply $T_*^\phi f(x)=\Norm{\vec{T}^\phi f(x)}{L^\infty}$. In place of the supremum norm, we now consider the stronger variational norms
\begin{equation*}
   \Norm{\vec{y}}{V^q}=\Norm{\{y_{\epsilon}\}_{\epsilon>0}}{V^q} :=\sup_{\{\epsilon_i\}_{i\in\Z}}\Big(\sum_{i=-\infty}^{\infty}\abs{y_{\epsilon_i}-y_{\epsilon_{i+1}}}^q\Big)^{1/q},
\end{equation*}
where the supremum is over all increasing sequences $\{\epsilon_i\}_{i\in\Z}$ of positive numbers. We define the smooth $q$-variation operator of $T$ as
\begin{equation*}
  V_q^\phi Tf(x):=\Norm{\vec{T}^\phi f(x)}{V^q}.
\end{equation*}

For certain classes of Calder\'on--Zygmund operators, norm inequalities for the non-linear operator $V_q^\phi T$, for $q>2$, are known in the literature; see Remark~\ref{rem:variation} for details. In the following result, we take one such estimate as a black-box assumption.

\begin{theorem}\label{thm:dyadicDominationVar}
Let $T$ be an $\omega$-Calder\'on--Zygmund operator, where the modulus of continuity $\omega$ satisfies the Dini condition~\eqref{eq:Dini}, and the smooth $q$-variation of $T$ satisfies the weak $(1,1)$ estimate:
\begin{equation}\label{eq:VqTweak11}
  \Norm{V_q^\phi Tf}{L^{1,\infty}}\lesssim\Norm{f}{L^1}.
\end{equation}
Let $Q_0$ be a cube which contains the support of $f$.
Then we have the pointwise bound
\begin{equation*}
    1_{Q_0}V_q^\phi Tf
    \lesssim Mf+\sum_{u\in\{0,\tfrac12,\tfrac13\}^d}\sum_{k=0}^{\infty}\omega(2^{-k})S^u_k\abs{f},
\end{equation*}
where $S^u_k$ is a positive dyadic shift of complexity $k$ with respect to the dyadic system
\begin{equation*}
  \mathscr{D}^u:=\{2^{-j}([0,1)^d+m+(-1)^j u):j\in\Z,m\in\Z^d\}.
\end{equation*}
\end{theorem}

\begin{remark}\label{rem:variation}
The weak-type bound for the sharp $q$-variation of $T$,
\begin{equation}\label{eq:VqTsharp}
  \Norm{V_q Tf}{L^{1,\infty}}\lesssim\Norm{f}{L^1},\qquad V_q Tf(x):=V_q^\chi Tf(x):=\Norm{\{T_\epsilon f(x)\}_{\epsilon>0}}{V^q},
\end{equation}
is known to hold at least for all Calder\'on--Zygmund convolution kernels under the following set of condition (see Campbell et al. \cite[Corollary 1.4]{CJRW:higher}):
\begin{equation*}
\begin{split}
  K(x,y)=k(x-y)=\frac{\Omega(x-y)}{\abs{x-y}^d},\qquad
  \Omega(x)=\Omega(\frac{x}{\abs{x}}),\qquad
  \Omega\in L^\infty(\mathbb{S}^{d-1}),\\
   \int_{\mathbb{S}^{d-1}}\Omega\ud\sigma=0,\qquad
  \int_{\abs{x}>2\abs{y}}\abs{k(x-y)-k(x)}\ud x\leq C\quad\forall y\in\R^d,
\end{split}
\end{equation*}
where the last-mentioned H\"ormander condition is a consequence of the Dini condition that we have imposed in any case. A particular case is the classical Hilbert transform \cite{CJRW:Hilbert}. The estimate \eqref{eq:VqTsharp} implies \eqref{eq:VqTweak11} when $\phi$ is a smooth radial cut-off.
\end{remark}

\begin{proof}[Proof of the last remark]
We identify $\phi(x)=\phi(\abs{x})$ as a function on $\R_+$. Then
\begin{equation*}
\begin{split}
  T_\epsilon^\phi f(x)
  &=\int\phi\Big(\frac{x-y}{\epsilon}\Big)K(x,y)f(y)\ud y
  =\int\int_0^{\abs{x-y}/\epsilon}\phi'(t)\ud t K(x,y)f(y)\ud y \\
  &=\int_0^{\infty}\phi'(t)\int_{\abs{x-y}>\epsilon t}K(x,y)f(y)\ud y\ud t
  =\int_0^{\infty}\phi'(t)T_{\epsilon t}f(x)\ud t.
\end{split}
\end{equation*}
Since $\Norm{\{y_{\epsilon t}\}_{\epsilon>0}}{V^q}=\Norm{\{y_{\epsilon}\}_{\epsilon>0}}{V_q}$, it follows that
\begin{equation*}
  V_q^\phi Tf(x)=\Norm{\vec{T}^\phi f(x)}{V_q}
  \leq\int_0^\infty\abs{\phi'(t)}\Norm{\vec{T}f(x)}{V_q}\ud t\lesssim V_q Tf(x).\qedhere
\end{equation*}
\end{proof}

We turn to the proof of Theorem~\ref{thm:dyadicDominationVar} and use again Lerner's formula \cite{Lerner:formula}:
\begin{equation*}
  1_{Q_0}\abs{V_q^\phi Tf-m_{V_q^\phi Tf}(Q_0)}\lesssim M_{\lambda,Q_0}^{\#}(V_q^\phi Tf)+\sum_{Q\in\mathcal{L}}\omega_{\lambda}(V_q^\phi Tf,Q^{(1)})1_Q,
\end{equation*}
where
\begin{equation*}
  \abs{m_{V_q^\phi Tf}(Q_0)}
  \lesssim\frac{1}{\abs{Q_0}}\Norm{V_q^\phi  Tf}{L^{1,\infty}}
  \lesssim\frac{1}{\abs{Q_0}}\Norm{f}{L^{1}}=\fint_{Q_0}\abs{f}\leq\inf_{Q_0} Mf,
\end{equation*}
using $V_q^\phi T:L^1\to L^{1,\infty}$ and $\supp f\subset Q_0$. Once we prove the following lemma, the rest of the proof proceeds exactly as the proof of Theorem~\ref{thm:dyadicDomination}:

\begin{lemma}
\begin{equation*}
  \omega_\lambda(V_q^\phi Tf,Q)\lesssim\sum_{k=0}^{\infty}\omega(2^{-k})\fint_{2^k Q}\abs{f}.
\end{equation*}
\end{lemma}

\begin{proof}
For $x\in Q$, we have
\begin{equation*}
  \vec{T}^\phi f(x)=\vec{T}^\phi (1_{2Q}f)(x)+\vec{T}^\phi (1_{(2Q)^c}f)(x)-\vec{T}^\phi (1_{(2Q)^c}f)(c_Q)+\vec{T}^\phi (1_{(2Q)^c}f)(c_Q),
\end{equation*}
and
\begin{equation*}
\begin{split}
  &\Norm{\vec{T}^\phi (1_{(2Q)^c}f)(x)-\vec{T}^\phi (1_{(2Q)^c}f)(c_Q)}{V^q} \\
   &=\BNorm{\sum_{k=1}^{\infty}\int_{2^{k+1}Q\setminus 2^k Q} [\vec{K}^\phi (x,y)-\vec{K}^\phi (c_Q,y)]f(y)\ud y}{V^q} \\
   &\leq\sum_{k=1}^{\infty}\int_{2^{k+1}Q\setminus 2^k Q} \Norm{\vec{K}^\phi (x,y)-\vec{K}^\phi (c_Q,y)}{V^q}\abs{f(y)}\ud y 
\end{split}
\end{equation*}
For any $\vec{u}=\{u_{\epsilon}\}_{\epsilon>0}$, we have
\begin{equation*}
  \Norm{\vec{u}}{V^q}\leq\Norm{\vec{u}}{V^1}\leq\int_0^\infty\Babs{\frac{\ud u_{\epsilon}}{\ud\epsilon}}\ud\epsilon.
\end{equation*}
We apply this to $\vec{u}=\vec{K}^\phi (x,y)-\vec{K}^\phi (c_Q,y)$:
\begin{equation*}
\begin{split}
  \partial_{\epsilon}K_\epsilon^\phi(x,y)
  =\partial_{\epsilon}[\phi((x-y)/\epsilon)K(x,y)]
  =(y-x)/\epsilon^2\cdot\nabla\phi((x-y)/\epsilon)K(x,y),
\end{split}
\end{equation*}
and hence
\begin{equation*}
\begin{split}
  \partial_{\epsilon} [K_\epsilon^\phi(x,y)-K_\epsilon^\phi(c_Q,y)]
  &=\Big(\frac{y-x}{\epsilon^2}-\frac{y-c_Q}{\epsilon^2}\Big)\cdot\nabla\phi\Big(\frac{x-y}{\epsilon}\Big)K(x,y) \\
   &\qquad+\frac{y-c_Q}{\epsilon^2}\cdot\Big[\nabla\phi\Big(\frac{x-y}{\epsilon}\Big)-\nabla\phi\Big(\frac{c_Q-y}{\epsilon}\Big)\Big]K(x,y) \\
    &\qquad+\frac{y-c_Q}{\epsilon^2}\cdot\nabla\phi\Big(\frac{c_Q-y}{\epsilon}\Big)[K(x,y)-K(c_Q,y)].
\end{split}
\end{equation*}
For $x\in Q$, $y\in 2^{k+1}Q\setminus 2^k Q$, using
\begin{equation*}
  \abs{K(x,y)}\lesssim\frac{1}{\abs{2^k Q}},\qquad
  \abs{K(x,y)-K(c_Q,y)}\lesssim\frac{\omega(2^{-k})}{\abs{2^kQ}},
\end{equation*}
as well as
\begin{equation*}
\begin{split}
  \Babs{\nabla\phi\Big(\frac{x-y}{\epsilon}\Big)} &\lesssim 1_{[\abs{x-y},2\abs{x-y}]}(\epsilon), \\
  \Babs{\nabla\phi\Big(\frac{x-y}{\epsilon}\Big)-\nabla\phi\Big(\frac{c_Q-y}{\epsilon}\Big)}
  &\lesssim\frac{\abs{x-c_Q}}{\epsilon} 1_{[\abs{x-y},2\abs{x-y}]\cup [\abs{c_Q-y},2\abs{c_Q-y}]}(\epsilon),
\end{split}
\end{equation*}
we deduce that
\begin{equation*}
  \int_0^\infty\abs{\partial_\epsilon[K_\epsilon^\phi(x,y)-K_\epsilon^\phi(c_Q,y)]}\ud\epsilon
  \lesssim\frac{\ell(Q)}{2^k\ell(Q)}\frac{1}{\abs{2^k Q}}+\frac{\omega(2^{-k})}{\abs{2^k Q}}
  \lesssim\frac{\omega(2^{-k})}{\abs{2^k Q}}.
\end{equation*}

Hence
\begin{equation*}
\begin{split}
  \sum_{k=1}^{\infty} &\int_{2^{k+1}Q\setminus 2^k Q}\Norm{\vec{K}^\phi (x,y)-\vec{K}^\phi (c_Q,y)}{V^q}\abs{f(y)}\ud y \\
  &\lesssim\sum_{k=1}^{\infty}\omega(2^{-k})\int_{2^{k+1}Q}\abs{f(y)}\ud y,
\end{split}
\end{equation*}
and thus
\begin{equation*}
\begin{split}
  \abs{V_q^\phi Tf(x)-V_q^\phi T(1_{(2Q)^c}f)(c_Q)}
  &\leq\Norm{\vec{T}^\phi f(x)-\vec{T}^\phi (1_{(2Q)^c}f)(c_Q)}{V^q} \\
  &\leq V_q^\phi T(1_{2Q}f)(x)+\sum_{k=1}^{\infty}  \omega(2^{-k}) \fint_{2^{k+1}Q}\abs{f}.
\end{split}
\end{equation*}
It follows that
\begin{equation*}
\begin{split}
  \omega_\lambda(T_* f;Q)
  &=\inf_c((V_q^\phi Tf -c)1_Q)^*(\lambda\abs{Q}) \\
  &\leq\Big(\big(V_q^\phi Tf-V_q^\phi T(1_{(2Q)^c}f)(c_Q)\big)1_Q\Big)^* (\lambda\abs{Q}) \\
  &\leq (V_q^\phi T(1_{2Q}f))^*(\lambda\abs{Q})+\sum_{k=1}^{\infty}  \omega(2^{-k}) \fint_{2^{k+1}Q}\abs{f},
\end{split}
\end{equation*}
where finally
\begin{equation*}
  (V_q^\phi T(1_{2Q}f))^*(\lambda\abs{Q})
  \leq\frac{1}{\lambda\abs{Q}}\Norm{V_q^\phi T(1_{2Q}f)}{L^{1,\infty}}
  \lesssim\frac{1}{\abs{Q}}\Norm{1_{2Q}f}{L^1}\lesssim\fint_{2Q}\abs{f}
\end{equation*}
by the assumption that $V_q^\phi  T$ satisfies a weak $(1,1)$ inequality.
\end{proof}

\section{Consequences for weighted norm inequalities}

It is clear that the two dyadic domination theorems can be used as a general machine to transfer results about dyadic shifts into results about (the maximal truncation or $q$-variation of) Calder\'on--Zygmund operators. We present just a couple of examples of this phenomenon for illustration.

We recall two recent results:

\begin{proposition}[\cite{HytPer}, Theorem 1.10]\label{prop:M}
\begin{equation*}
  \Norm{M(f\sigma)}{L^p(w)}\lesssim\big([w,\sigma]_{A_p}[\sigma]_{A_\infty}\big)^{1/p}\Norm{f}{L^p(\sigma)}.
\end{equation*}
\end{proposition}

\begin{proposition}[\cite{HytLac}, Propositions 3.2 and 5.3]\label{prop:S}
Let $S$ be a positive dyadic shift of complexity $k$. Then
\begin{equation*}
  \Norm{S(f\sigma)}{L^p(w)}\lesssim(1+k)[w,\sigma]_{A_p}^{1/p}\big([w]_{A_\infty}^{1/p'}+[\sigma]_{A_\infty}^{1/p}\big)\Norm{f}{L^p(\sigma)}.
\end{equation*}
\end{proposition}

In combination with Theorem~\ref{thm:dyadicDomination}, we get the following corollary. It is essentially known, but we obtain it for a slightly bigger class of Calder\'on--Zygmund kernels with a rather general modulus of continuity.

\begin{corollary}\label{cor:Ap}
Let $T$ be an $\omega$-Calder\'on--Zygmund operator, where the modulus of continuity satisfies
\begin{equation}\label{eq:DiniLog}
  \int_0^1\omega(t)\log\frac{1}{t}\frac{\ud t}{t}<\infty.
\end{equation}
Then for all $w,\sigma\in A_\infty$ and $p\in(1,\infty)$, we have
\begin{equation*}
   \Norm{T_*(f\sigma)}{L^p(w)}\lesssim [w,\sigma]_{A_p}^{1/p}\big([w]_{A_\infty}^{1/p'}+[\sigma]_{A_\infty}^{1/p}\big)\Norm{f}{L^p(\sigma)},
\end{equation*}
and in particular
\begin{equation*}
  \Norm{T_*f}{L^p(w)}\lesssim[w]_{A_p}^{\max\{1,1/(p-1)\}}\Norm{f}{L^p(w)}.
\end{equation*}
\end{corollary}

\begin{proof}
By Theorem~\ref{thm:dyadicDomination}, we have
\begin{equation*}
  \Norm{1_{Q_0}T_*(f\sigma)}{L^p(w)}
  \lesssim\Norm{M(f\sigma)}{L^p(w)}+\sum_{u\in\{0,\tfrac13,\tfrac23\}^d}\sum_{k=0}^{\infty}\omega(2^{-k})
      \Norm{S^u_k(f\sigma)}{L^p(w)}.
\end{equation*}
It suffices to apply Proposition~\ref{prop:M} to $M$ and Proposition~\ref{prop:S} to each $S^u_k$, and to observe that
\begin{equation*}
  \sum_{k=0}^{\infty}\omega(2^{-k})k\eqsim\int_0^1\omega(t)\log\frac{1}{t}\frac{\ud t}{t}<\infty.
\end{equation*}
Finally, we let $Q_0\uparrow\R^d$, and use monotone convergence.
\end{proof}

The following corollary to Theorem~\ref{thm:dyadicDominationVar} is completely new, and extends the sharp weighted estimates to the $q$-variation of singular integrals for the first time.

\begin{corollary}\label{cor:ApVar}
Let $T$ be an $\omega$-Calder\'on--Zygmund operator, where the modulus of continuity satisfies \eqref{eq:DiniLog}, and suppose that the smooth $q$-variation of $T$ satisfies the weak $(1,1)$ bound \eqref{eq:VqTweak11}.
Then for all $w,\sigma\in A_\infty$ and $p\in(1,\infty)$, we have
\begin{equation*}
   \Norm{V_q^\phi T(f\sigma)}{L^p(w)}\lesssim [w,\sigma]_{A_p}^{1/p}\big([w]_{A_\infty}^{1/p'}+[\sigma]_{A_\infty}^{1/p}\big)\Norm{f}{L^p(\sigma)},
\end{equation*}
and in particular
\begin{equation*}
    \Norm{V_q^\phi Tf}{L^p(w)}\lesssim[w]_{A_p}^{\max\{1,1/(p-1)\}}\Norm{f}{L^p(w)}.
\end{equation*}
\end{corollary}

\begin{proof}
Repeat the proof of Corollary~\ref{cor:Ap} with obvious changes.
\end{proof}

\bibliography{weighted}

\begin{thebibliography}{1}

\bibitem{CJRW:Hilbert}
James~T. Campbell, Roger~L. Jones, Karin Reinhold, and M{\'a}t{\'e} Wierdl.
\newblock Oscillation and variation for the {H}ilbert transform.
\newblock {\em Duke Math. J.}, 105(1):59--83, 2000.

\bibitem{CJRW:higher}
James~T. Campbell, Roger~L. Jones, Karin Reinhold, and M{\'a}t{\'e} Wierdl.
\newblock Oscillation and variation for singular integrals in higher
  dimensions.
\newblock {\em Trans. Amer. Math. Soc.}, 355(5):2115--2137 (electronic), 2003.

\bibitem{Hytonen:A2}
Tuomas Hyt\"onen.
\newblock The sharp weighted bound for general {Calder\'on-Zygmund} operators.
\newblock {\em Ann. of Math.}
\newblock (to appear). Preprint, arXiv:1007.4330 (2010).

\bibitem{HytLac}
Tuomas Hyt\"onen and Michael~T. Lacey.
\newblock The {$A_p-A_\infty$} inequality for general {Calder\'on--Zygmund}
  operators.
\newblock {\em Indiana Univ. Math. J.}
\newblock (to appear). Preprint, arXiv:1106.4797 (2011).

\bibitem{HytPer}
Tuomas Hyt\"onen and Carlos P\'erez.
\newblock Sharp weighted bounds involving {$A_\infty$}.
\newblock {\em Anal. PDE}.
\newblock (to appear). Preprint, arXiv:1103.5562 (2011).

\bibitem{Lerner:posDyad}
Andrei~K. Lerner.
\newblock On an estimate of {Calder\'on--Zygmund} operators by dyadic positive
  operators.
\newblock Preprint, arXiv:1202.1860 (2012).

\bibitem{Lerner:formula}
Andrei~K. Lerner.
\newblock A pointwise estimate for the local sharp maximal function with
  applications to singular integrals.
\newblock {\em Bull. Lond. Math. Soc.}, 42(5):843--856, 2010.

\bibitem{NTV:Tb}
Fedja Nazarov, Sergei Treil, and Alexander Volberg.
\newblock The {$Tb$}-theorem on non-homogeneous spaces.
\newblock {\em Acta Math.}, 190(2):151--239, 2003.

\end{thebibliography}
\bibliographystyle{plain}

\end{document}